\documentclass{amsart}

\usepackage{thesis} 

\theoremstyle{plain}

\newtheorem*{slogan*}{Slogan}

\newcommand{\tmf}{tm\!f}
\newcommand{\mmf}{mm\!f}
\renewcommand{\ell}{{\textup{ell}}}

\newcommand{\bbC}{\mathbb{C}}
\newcommand{\bbG}{\mathbb{G}}
\newcommand{\bbP}{\mathbb{P}}

\newcommand{\mot}{{\textup{mot}}}
\newcommand{\cell}{{\textup{cell}}}

\newcommand{\oO}{\mc{O}}

\newcommand{\spf}{{\textup{Spf}}}

\begin{document}

\title{$\bbE_\infty$ automorphisms of motivic Morava $E$-theories}
\date{\today}
\author{\textsc{Aaron Mazel-Gee}}

\begin{abstract}
We apply Goerss--Hopkins obstruction theory for motivic spectra to study the motivic Morava $E$-theories.  We find that they always admit $\bbE_\infty$ structures, but that these may admit ``exotic'' $\bbE_\infty$ automorphisms not coming from the usual Morava stabilizer group.
\end{abstract}

\maketitle




\setcounter{section}{-1}

\section{Introduction}

\subsection{Overview}

In this short note, we equip the motivic Morava $E$-theory spectra with canonical $\bbE_\infty$ structures, and compute their automorphisms as $\bbE_\infty$ ring spectra.  We find that these automorphism groups are (homotopically) discrete, but that they are apparently distinct from the usual Morava stabilizer group.  We refer the reader to \Cref{thm motivic morava E-theories} for a precise statement of our main result, and to \Cref{rem compare mot morava Ethy result with aut of FGLs} for a discussion of these automorphism groups.  In \Cref{remark compare with NSO on KGL}, we explain the precise relationship between our work and that of Naumann--Spitzweck--{\O}stv{\ae}r \cite{NSOAKT} on motivic algebraic K-theory (i.e.\! in the height-1 case).

Our proof is patterned directly on that of Goerss--Hopkins \cite{GH-moduli-spaces,GH} for the ordinary (i.e.\! non-motivic) Morava $E$-theory spectra (which is based on much prior work, notably that of Hopkins--Miller \cite{RezkHopMil}).  Whereas their proof is based in Goerss--Hopkins obstruction theory for ordinary spectra, our proof uses our generalization \cite{GHOsT} of Goerss--Hopkins obstruction theory to an arbitrary presentably symmetric monoidal stable $\infty$-category.

The most immediate consequence of the present work is that it endows the motivic cohomology theories represented by the motivic Morava $E$-theories with the rich algebraic structure of power operations.  However, we also view it as a first step towards a moduli-theoretic construction of a motivic spectrum $\mmf$ of \textit{motivic modular forms}, in analogy with the ordinary spectrum $\tmf$ of topological modular forms \cite{tmfbook}.\footnote{The works \cite{Ricka-mmf,GIKR-mmf} take a different approach, producing motivic spectra over $\bbR$ and $\bbC$ whose cohomologies coincide with that expected of $\mmf$ (in analogy with $\tmf$).  These constructions are indirect, and relatively specific to the chosen base fields; in particular, the resulting motivic spectra are not manifestly related to any theory of elliptic motivic spectra.}  As the construction of $\tmf$ has been highly influential in chromatic homotopy theory, so would the construction of $\mmf$ significantly advance the chromatic approach to motivic homotopy theory, which is a highly active area of research \cite{Voe-ICM,HuKriz-MGL,Vezz-BP,Borgh-K,Horn-chrommot,LevMor-MGL,PPR-univ,NSO,NSO-nonregular,BalmerSpectra,Isaksen-wantmmf,Isaksen-htpyofCmmf,Andrews-families,Hoyois-MGL,Horn-nilp,Joachimi-thick,HO-PrimesMot,Gheorge-exotic}.

There has been much recent interest in ``genuine'' operadic structures, e.g.\! genuine $G$-spectra with multiplications indexed by maps of finite $G$-sets (instead of just finite sets) \cite{BH-Ninfty,HH-Eqmult,BH-Gsm,BH-Tamb,Rubin-Ninfty,BP-eqop,GW-eqop}, as well as analogous structures in motivic homotopy theory \cite{BachHoy}.  We do not contend with such structures here.  However, we are optimistic that our generalization of Goerss--Hopkins obstruction theory admits a fairly direct enhancement to one that would handle them in a formally analogous way.  Thereafter, it seems quite plausible that the present work would admit a straightforward extension to give ``motivically genuine'' $\bbE_\infty$ structures on the motivic Morava $E$-theory spectra.

\subsection{Conventions}

\begin{itemize}

\item We write $\Sp^\mot$ for the (presentably symmetric monoidal stable) $\infty$-category of motivic spectra.\footnote{We implicitly work over a regular noetherian base scheme of finite Krull dimension, but this is only in order to employ the results of \cite{NSO}.  We will additionally use a result of \cite{GepSnaithSplitting}, which requires a (not necessarily regular) noetherian base scheme of finite Krull dimension.}  This comes equipped with a distinguished group of invertible objects
\[ \G = \{ S^{i,j} \}_{i,j \in \bbZ} \cong \bbZ \times \bbZ , \]
the \textit{motivic sphere spectra}: the unit object is $\unit = S^{0,0}$, its categorical suspension is $\Sigma \unit = S^{1,0}$, and then by definition we have $\Sigma^\infty \bbG_m = S^{1,1}$.  In particular, it follows that $S^{2,1} = \Sigma^\infty \bbP^1$.

\item For any $X \in \Sp^\mot$, we write $X_{**} = \pi_{**} X$ for its bigraded homotopy groups, i.e.\! $X_{i,j} = \pi_{i,j} X = [S^{i,j},X]_{\Sp^\mot}$.  Additionally, we write $X_* = \pi_* X$ for its $(2,1)$-line of homotopy groups, i.e.\! $X_i = \pi_i X = [S^{2i,i} , X]_{\Sp^\mot}$.

\item We write $\Sp^\mot_\cell \subset \Sp^\mot$ for the coreflective subcategory of \textit{cellular} motivic spectra.  This is the subcategory generated under colimits by the motivic sphere spectra.  It can also be characterized as the subcategory of colocal objects for the ``bigraded homotopy groups'' functor; in particular, within this subcategory, bigraded homotopy groups detect equivalences.

\item We fix a finite field $k$ of characteristic $p > 0$, and we fix a formal group law $\bbG_0$ over $k$ of finite height $n$.

\item We write $E(k,\bbG_0)$ for the corresponding Lubin--Tate deformation ring, we write $\mf{m} \subset E(k,\bbG_0)$ for its unique maximal ideal, and we fix an isomorphism $E(k,\bbG_0)/\mf{m} \cong k$.

\item We fix a versal deformation $\bbG$ of $\bbG_0$ over $E(k,\bbG_0)$.  To be precise, $\bbG$ is a formal group law over $E(k,\bbG_0)$, and pushes forward to $\bbG_0$ along the now-canonical map $E(k,\bbG_0) \ra k$.  Geometrically, this corresponds to a pullback
\[ \begin{tikzcd}
\bbG_0 \arrow{r} \arrow{d} & \bbG \arrow{d} \\
\spec(k) \arrow{r} & \spf(E(k,\bbG_0))
\end{tikzcd} \]
of formal groups (where we notationally identify formal group laws with their underlying formal groups).

\item We write \[ E^\top = E^\top_{k,\bbG_0} \in \Sp \] for the (ordinary) Morava $E$-theory spectrum corresponding to the pair $(k,\bbG_0)$, coming from the Landweber exact functor theorem (see e.g.\! \cite[Theorem 6.4 and 6.9]{RezkHopMil}) applied to the formal group law $\bbG$ over $E(k,\bbG_0)$.\footnote{This is known to be $\bbE_\infty$, by \cite[Corollary 7.6]{GH-moduli-spaces} (which is precisely the result we generalize here).}  To be precise, we have a chosen isomorphism \[ E_*^\top \cong E(k,\bbG_0)[u^\pm] \] (with $|u|=2$), and the degree-$(-2)$ formal group law $\ol{\bbG}$ on $E_*^\top$ coming from its complex orientation corresponds to $\bbG$ via the unit $u \in E_2^\top$, considered as a degree-$0$ formal group law on $E_*^\top$.

\item We write \[ E = E^\mot = E^\mot_{k,\bbG_0} \in \Sp^\mot_\cell \] for the motivic Morava $E$-theory spectrum corresponding to the pair $(k,\bbG_0)$ coming from the motivic Landweber exact functor theorem of \cite[Theorem 8.7]{NSO}.  This is cellular by construction, and comes equipped with a quasi-multiplication (i.e.\! a multiplication up to phantom maps).  Moreover, writing $MGL \in \Sp^\mot$ for the algebraic bordism spectrum and $MU \in \Sp$ for the complex bordism spectrum, we have isomorphisms \[ E_{**} \cong MGL_{**} \otimes_{MU_*} E_*^\top \] and \[ E_{**} E \cong E_{**} \otimes_{E_*^\top} E_*^\top E^\top, \] and the structure maps of the Hopf algebroid $(E_{**},E_{**}E)$ are tensored up from those of $(E_*,E_*E)$.

\end{itemize}

\subsection{Acknowledgments}

David Gepner was instrumental in deducing this application of $\infty$-categorical Goerss--Hopkins obstruction theory, and it is a pleasure to acknowledge his help.  We would also like to acknowledge Markus Spitzweck for his helpful correspondence, as well as the NSF graduate research fellowship program (grant DGE-1106400) for financial support during the time that this research was carried out.

\section{$\bbE_\infty$ automorphisms of motivic Morava $E$-theories}

We now state the main result.

\begin{thm}\label{thm motivic morava E-theories}
The motivic Morava $E$-theory spectrum $E = E^\mot_{k,\bbG_0}$ has a unique $\bbE_\infty$ structure refining the ring structure on its bigraded homotopy groups, and as such generates a subgroupoid of $\CAlg(\Sp^\mot)$ equivalent to
\[ B(\Aut_{\CAlg(\Comod_{(E_{**},E_{**}E)})}(E_{**}E)) . \]
In particular, its space of automorphisms is discrete.
\end{thm}

\begin{lem}
Any Landweber exact motivic spectrum satisfies Adams's condition.
\end{lem}

\begin{proof}
The proof is almost identical to that of \cite[Proposition 15.3]{RezkHopMil}.  First of all, the general statement follows from the universal case of $MGL$.  In turn, we can present $MGL$ as a filtered colimit of Thom spectra over finite Grassmannians, which are then dualizable.  Let us write this as $MGL \simeq \colim_\alpha MGL_\alpha$.\footnote{Explicitly, $\Dual (MGL_\alpha)$ is also a Thom spectrum via the formula $\Dual (X^\xi) \simeq X^{-\xi}$.}  So, it only remains to verify that $MGL_{**}(\Dual (MGL_\alpha))$ is projective as an $MGL_{**}$-module.  In bidegree $(0,0)$, we observe that $MGL_{**}(\Dual(MGL_\alpha)) \cong (MGL^{**}MGL_\alpha)^\vee$, so that here the claim follows from the algebra presentation of \cite[Proposition 2.19]{GepSnaithSplitting}, which in particular implies (by inducting on the dimension of the Grassmannians) that this algebra itself is actually free as an $MGL_{**}$-module.  From here, in an arbitrary bidegree $(i,j)$ we then compute that
\begin{align*}
MGL_{i,j}(\Dual (MGL_\alpha)) & \cong MGL_{0,0}(S^{-i,-j} \otimes \Dual(MGL_\alpha)) \\
& \cong MGL_{0,0}(S^{-i,-j}) \otimes_{MGL_{0,0}} MGL_{0,0}(\Dual(MGL_\alpha))
\end{align*}
(using the K\"unneth theorem).
\end{proof}

\begin{obs}\label{two types of localization}
By definition, $E_{**}$-localization in $\Sp^\mot$ is the localization determined by the $E_{**}$-acyclics, i.e.\! those objects $Z$ such that $E_{**} Z \cong 0$.  Note that such motivic spectra $Z$ may not be $E$-acyclic, i.e.\! it might still be the case that $E \otimes Z \not\simeq 0$.  On the other hand, if $Z$ is also cellular, since $E$ is cellular then so is $E \otimes Z$ (since $\Sp^\mot_\cell$ is a colocalization of $\Sp^\mot$ and the symmetric monoidal structure commutes with colimits in each variable).  Thus, when restricted to cellular motivic spectra, the localizations $\leftloc_E$ and $\leftloc_{E_{**}}$ agree.  This is summarized by the diagram
\[ \begin{tikzcd}
L_{E_{**}}(\CAlg(\Sp^\mot_\cell)) \arrow{r} \arrow{d}[sloped, anchor=north]{\sim} & L_{E_{**}}(\CAlg(\Sp^\mot)) \arrow{d} \\
L_E(\CAlg(\Sp^\mot_\cell)) \arrow{r} \arrow{d} & L_E(\CAlg(\Sp^\mot)) \arrow{d} \\
\CAlg(\Sp^\mot_\cell) \arrow{r} & \CAlg(\Sp^\mot)
\end{tikzcd} \]
of $\infty$-categories.
\end{obs}

\begin{proof}[Proof of \Cref{thm motivic morava E-theories}]
The proof is formally identical to that of \cite[Corollary 7.6]{GH-moduli-spaces}, only we work in the $\infty$-category $\Sp^\mot_\cell$: the key pieces of input are \cite[Theorems 8.5, 8.8, and 8.9]{GHOsT}, which are respectively generalizations of \cite[Proposition 5.2, Proposition 5.5, and Theorem 5.8]{GH-moduli-spaces}.  The passage from the ordinary case to the motivic case runs as follows.

First of all, a priori we only have a quasi-multiplication on $E \in \Sp^\mot_\cell$.  However, this suffices to give all the required structure on its bigraded $E$-homology groups: these are by definition homotopy classes of maps out of bigraded spheres, which by definition cannot detect phantom maps.

Next, a priori, Goerss--Hopkins obstruction theory in $\Sp^\mot_\cell$ using the homology theory $E_{**}$ computes a moduli space in $\leftloc_{E_{**}}(\CAlg(\Sp^\mot_\cell))$.  However, as explained in \Cref{two types of localization}, we have an equivalence
\[ \leftloc_{E_{**}}(\CAlg(\Sp^\mot_\cell)) \simeq \leftloc_E(\CAlg(\Sp^\mot_\cell)) , \]
and the usual proof that $E$ is $E$-local then applies (see e.g.\! \cite[Proposition 1.17]{RavLoc}).  Thus we have $E \in \leftloc_{E_{**}}\Sp^\mot_\cell$, and hence the moduli space that we construct inside of $\CAlg(\leftloc_{E_{**}}(\Sp^\mot_\cell)) \simeq \leftloc_{E_{**}}(\CAlg(\Sp^\mot_\cell))$ is that of an object whose underlying motivic spectrum is indeed $E$ itself.

Now, let us turn to the remainder of the proof of \cite[Corollary 7.6]{GH-moduli-spaces} and its ingredients.  We do \textit{not} carry over the last line (which identifies the relevant automorphism group with an automorphism group in a category of formal group laws).\footnote{However, see \Cref{rem compare mot morava Ethy result with aut of FGLs}.}  However, everything else used there is entirely algebraic, and works equally well replacing ordinary gradings with bigradings.  Note that the gradings appearing in \cite[\sec 6]{GH-moduli-spaces} arise from the external simplicial direction (and the internal gradings play no real role); note too that the ``Dyer--Lashof operations'' arising there arise from the algebraic theory given in \cite{MaySteenrod} (and in particular have nothing whatsoever to do with operations in motivic homology).
\end{proof}

\begin{rem}\label{rem compare mot morava Ethy result with aut of FGLs}
Using various adjunctions as well as the fact that all morphisms respect bigradings, one can identify the endomorphism monoid
\[ \End_{\CAlg(\Comod_{(E_{**},E_{**}E)})}(E_{**}E) \]
(the classifying space of whose maximal subgroup appears in the statement of \Cref{thm motivic morava E-theories}) with the hom-set
\[ \hom_{\CAlg(\Mod_{E_*^\top})}(E_*^\top E^\top , MGL_* \otimes_{MU_*} E_*^\top) . \]
This appears to fall under the auspices of \cite[\sec 17]{RezkHopMil}, and thus ought to have a moduli-theoretic interpretation.

A reasonable guess would be that, if we define the map $\chi$ via the pullback diagram
\[ \begin{tikzcd}
\spec(E_*^\mot) \arrow{r}{\chi} \arrow{d} & \spec(E_*^\top) \arrow{d} \\
\spec(MGL_*) \arrow{r} & \spec(MU_*) ,
\end{tikzcd} \]
then the group in question should be the group of (strict) automorphisms of the formal group law $\chi^* \bbG$ over \[ E_*^\mot = MGL_* \otimes_{MU_*} E_*^\top.\]  However, we have not managed to verify this claim.  If it holds, however, it would be in keeping with the general philosophy that motivic homotopy theory should be thought of as a flavor of parametrized homotopy theory: the pullback of a sheaf over a small space to a larger one will generally admit more automorphisms than the original sheaf itself.

In any case, there is an evident map to this automorphism group from the Morava stabilizer group, which therefore acts on the object $E^\mot \in \CAlg(\Sp^\mot)$ as well.  Moreover, this map should be an inclusion whenever the map $MU_* \ra MGL_*$ is (indeed, in certain cases the latter is even an isomorphism (see \cite{HoyoisMGL})).
\end{rem}

\begin{rem}
\label{remark compare with NSO on KGL}
in \cite{NSOAKT}, Naumann--Spitzweck--{\O}stv{\ae}r prove that the motivic algebraic K-theory spectrum $KGL$ (over a noetherian base scheme of finite Krull dimension) admits a unique $\bbE_\infty$ structure refining the canonical multiplication on its represented motivic cohomology theory.  Meanwhile, Goerss--Hopkins obstruction theory takes a commutative algebra in comodules and returns the moduli space of realizations.  These are not directly comparable: the former addresses the question of $\bbE_\infty$ structures on a \textit{given} object, while the latter addresses the question of the $\infty$-groupoid of objects that \textit{realize} some chosen algebraic datum.  Moreover, \cite{NSOAKT} addresses $KGL$ as an integral object, whereas \Cref{thm motivic morava E-theories} only applies to $\E^\mot_{k,\hat{\bbG}_m} \simeq KGL^\sm_p$.

To clarify, for a variable object $X \in \Sp^\mot_\cell$ we locate both the main theorem of \cite{NSOAKT} as well as \Cref{thm motivic morava E-theories} in the diagram
\[ \begin{tikzcd}
\hom_\Op(\Comm,\E\textup{nd}_{\Sp^\mot_\cell}(X)) \arrow{d} \\
\hom_\Op(\Comm,\E\textup{nd}_{\ho(\Sp^\mot_\cell)}(X))) \arrow{r}{E_{**}} & \CAlg(\Comod_{(E_{**},E_{**}E)})^\simeq \arrow{d}{\ms{M}(-)} \\
& \S_{/\leftloc_{E_{**}}(\CAlg(\Sp^\mot_\cell))}
\end{tikzcd} \]
(where $\E\textup{nd}$ denotes the endomorphism operad): the two downwards arrows are the settings for the respective theorems.
\begin{itemize}
\item On the one hand, taking $X = KGL$, there is a canonical point in the set $\hom_\Op(\Comm,\E\textup{nd}_{\ho(\Sp^\mot_\cell)}(KGL))$ which selects the standard multiplication on $KGL$ in $\ho(\Sp^\mot_\cell)$.  The main theorem of \cite{NSOAKT} can then be interpreted as saying that the fiber over this point is nonempty and contractible.
\item On the other hand, Goerss--Hopkins obstruction takes an \textit{algebraic} object in $\CAlg(\Comod_{(E_{**},E_{**}E)})^\simeq$ and provides a spectral sequence converging to the homotopy groups of its moduli space of realizations (which in our case collapses), considered as a subgroupoid of the $\infty$-category $\leftloc_{E_{**}}(\CAlg(\Sp^\mot_\cell))$.  The inclusion of this subgroupoid is the target of this algebraic object under the lower vertical map.
\end{itemize}

A toy example illustrating the difference between these two approaches is the difference between $\bbE_\infty$ structures on a fixed two-element set (of which there are four) and the moduli space of such objects in $\CAlg(\Set)$ (which consists of two discrete components).\footnote{However, this analogy fails in that the upper vertical map is already an equivalence since $\Set \xra{\sim} \ho(\Set)$.}  These two approaches are both explored in the more sophisticated setting of algebras over an operad in \cite{Rezkthesis}.

Note that the horizontal map in this diagram may not be injective: it is a priori possible that distinct multiplications on $X$ in $\ho(\Sp^\mot_\cell)$ might induce the same commutative algebra object structure on $E_{**}X \in \Comod_{(E_{**},E_{**}E)}$.  This represents a further obstruction to a direct comparison of these two approaches to the realization problem.
\end{rem}

\bibliographystyle{amsalpha}
\bibliography{Enmot}{}

\end{document}